\documentclass[12pt]{amsart}
\usepackage{amssymb,amsmath,amsthm,comment}
\usepackage{pdfpages,graphicx} 
\usepackage[section]{placeins} 
\usepackage{wrapfig}
\usepackage{float}
\usepackage{lineno}
\usepackage{youngtab}
\usepackage{setspace}
\newcommand{\shrinkmargins}[1]{
  \addtolength{\textheight}{#1\topmargin}
  \addtolength{\textheight}{#1\topmargin}
  \addtolength{\textwidth}{#1\oddsidemargin}
  \addtolength{\textwidth}{#1\evensidemargin}
  \addtolength{\topmargin}{-#1\topmargin}
  \addtolength{\oddsidemargin}{-#1\oddsidemargin}
  \addtolength{\evensidemargin}{-#1\evensidemargin}
  }
\shrinkmargins{0.5}

\newtheorem{theorem}{Theorem}
\newtheorem{lemma}[theorem]{Lemma}

\newtheorem{corollary}[theorem]{Corollary}
\newtheorem*{theorem*}{Theorem}

\newtheorem{conjecture}{Conjecture}
\newtheorem{proposition}[theorem]{Proposition}
{Claim}

\newtheorem*{definition}{Definition}

\theoremstyle{remark}

\newtheorem*{remarks}{{\bf Remarks}}

\numberwithin{theorem}{section} \numberwithin{equation}{section}

\usepackage{caption}
\usepackage{multirow}

\def\func#1{\mathop{\rm #1}}%

\begin{document}
\title[Plane Partitions]{Inequalities for Plane Partitions}
\author{Bernhard Heim }
\address{Lehrstuhl A f\"{u}r Mathematik, RWTH Aachen University, 52056 Aachen, Germany}
\email{bernhard.heim@rwth-aachen.de}
\author{Markus Neuhauser}
\address{Kutaisi International University, 5/7, Youth Avenue,  Kutaisi, 4600 Georgia}
\email{markus.neuhauser@kiu.edu.ge}
\author{Robert Tr{\"o}ger}
\email{robert@silva-troeger.de}
\begin{abstract}
Inequalities are important features in the context
of sequences of numbers and polynomials.
The Bessenrodt--Ono inequality  for
partition numbers and Nekrasov--Okounkov polynomials 
has only recently been discovered. In this paper we study the
log-concavity (Tur\'{a}n inequality) and Bessenrodt--Ono inequality for
plane partitions and their polynomization.
\end{abstract}
\maketitle
\section{Introduction and Main Results}
In this paper we address inequalities for plane partitions and
their polynomization.
Plane partitions are, according to Stanley, fascinating generalizations of 
partitions of integers (\cite{St99}, Section 7.20). Andrews \cite{An98} gave an excellent introduction
of plane partitions in the context of higher-dimensional partitions. We also refer to Krattenthaler's
survey on plane partitions in the work of Stanley and his school
\cite{Kr16}.

A plane partition $\pi$ of $n$ is an array 
$\pi = \left( \pi_{ij} \right)_{i,j \geq 1}$ of non-negative integers $\pi_{ij}$
with finite sum $\vert \pi \vert := \sum_{i,j=1} \pi_{ij}=n$, which is weakly decreasing in rows and columns.
It can be considered as the filling of a Ferrers diagram with weakly decreasing rows and columns,
where the sum of all these numbers is equal to $n$.
Let the numbers in the filling represent the 
heights for stacks of blocks placed on each cell of the diagram (Figure \ref{cube}). 
This is a natural generalization of the concept of
classical partitions \cite{An98, On03}.

\begin{figure}[H]

\begin{minipage}{0.35\textwidth}
\begin{equation*}
\phantom{xxx}
\young(5443321,432,21) 
\end{equation*}
\end{minipage}
$\longrightarrow$ \phantom{xx}
\begin{minipage}{0.4\textwidth}\phantom{xx}
\includegraphics[width=0.75\textwidth]{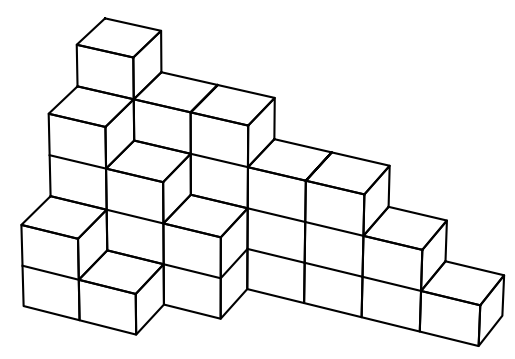}  
\end{minipage}
\caption{\label{cube}Representation of plane partitions.}
\end{figure}

Let ${p}(n)$ denote the number of partitions of $n$ and $\func{pp}(n)$ the number of plane partitions of $n$.
As usual, ${p}\left( 0\right):=1$ and $\func{pp} \left(0 \right): =1$.
In Table \ref{partition} we have listed the first values
of the partition and plane partition function.
\begin{table}[H]
\[
\begin{array}{cccccccccccc}
\hline
n & 0 & 1 & 2 & 3 & 4 & 5 & 6 & 7 & 8 & 9 & 10 \\ \hline \hline
p\left( n \right) & 1 & 1 & 2 &
3 &
5 &
7 &
11 &
15 &
22 & 30 & 42 \\
\func{pp} \left( n \right) & 1 & 1 & 3 &
6 &
13 &
24&
48 &
86 &
160 & 282 & 500 \\
\hline
\end{array}
\]
\caption{\label{partition}
Values for $ 0 \leq n \leq 10$.}
\end{table}
We investigate
the log-concavity \cite{Br89, St89} and the
Bessenrodt--Ono inequality \cite{BO16, HNT20} of plane partitions.
A sequence of real, non-negative numbers $\{\alpha(n)\}_{n=0}^{\infty}$ is log-concave if for all $n \in \mathbb{N}$:
\begin{equation}\label{lc}
\alpha(n)^2 > \alpha(n-1) \, \alpha(n+1).
\end{equation}
We say log-concave at $n$ if (\ref{lc}) is satisfied for a specific $n$.
A sequence of real non-negative numbers $\{\alpha(n)\}_{n=0}^{\infty}$ satisfies the
Bessenrodt--Ono inequality 
for a set $S \subset \mathbb{N} \times \mathbb{N}$,
if for all $(a,b) \in S$:
\begin{equation}\label{bo}
\alpha(a) \, \alpha(b) > \alpha(a+b).
\end{equation}
\subsection{Related Work and Recent Results}
We recall that Nicolas \cite{Ni78} proved that the partition function is log-concave for $n > 25$:
\begin{equation}\label{lc:partition}
{p}(n)^2 > {p}(n-1) \, {p}(n+1).
\end{equation}
It is valid for all $n$ even and fails for $ 1 \leq n \leq 25$ odd.
In the course of proving several conjectures of Chen and Sun, DeSalvo and Pak \cite{DP15}
reproved this result. They remarked that due to the
Hardy--Ramanujan asymptotic formula
\[
p(n) \sim \frac{1}{4 \sqrt{3n}} \, e^{\pi \sqrt{\frac{2}{3}n}} \quad \text{ as } n
\rightarrow \infty,
\]
there is no way of knowing precisely when the asymptotic formula dominates the calculation.
Their proof is based on Rademacher type estimates
\cite{Ra37}
by Lehmer
(e.~g.\ \cite{Le38,Le39}), which provide the
demanded, explicit, guaranteed error estimate. This proves the log-concavity for $n \geq 2,600$
(\cite{DP15}). More generally, let $p_k(n)$ be the $k$-colored partition function, obtained for
every $k \in \mathbb{N}$ by the generating function $$\prod_{n=1}^{\infty} \left(1 - q^n \right)^{-k}= \sum_{n=0}^{\infty} p_k(n) \, q^n,$$
which is essentially the $k$th
power of the reciprocal of the Dedekind eta function \cite{On03}.
Then (\ref{lc:partition}) was extended by Chern--Fu--Tang \cite{CFT18} to an interesting conjecture:
Suppose $k,\ell ,n \in \mathbb{N}$ with $k \geq 2$ and
$n >\ell $. Let $\left( k,n,\ell \right) \neq (2,6,4)$, then
\[
p_k(n-1) \, p_{k}\left( \ell +1\right) \geq p_k(n) \, p_{k}\left( \ell \right) .
\]
The conjecture was extended to $k \in \mathbb{R}_{\geq 2}$ \cite{HN21A}, which involves the
so-called D'Arcais polynomials or Nekrasov--Okounkov
polynomials \cite{NO06, Ha10}.
The Conjecture by Chern--Fu--Tang
and some portion of the Conjecture by Heim--Neuhauser
was recently proven by Bringmann, Kane, Rolen, and Tripp \cite{BKRT21}.
The proof is based on Rademacher type formulas for the coefficients of 
powers of the Dedekind eta function, utilizing the weak modularity property.

Bessenrodt and Ono \cite{BO16} discovered a beautiful and simple inequality for
partition numbers.
Let $a,b \in \mathbb{N}$. Suppose $a,b \geq 2$ and $a+b \geq 10$, then
\begin{equation}\label{bo: partition}
p(a) \, p(b) > p(a+b).
\end{equation}
The inequality is symmetric and always fails for $a=1$ or $b=1$. Let $2 \leq b \leq a$.
There is equality for the pairs $(4,3),(6,2),(7,2)$ and the 
opposite inequality of
(\ref{bo: partition}) is exactly true
for the pairs $(2,2), (3,2), (4,2), (5,2), (3,3), (5,3)$. Bessenrodt and Ono's proof
is based on
an
analytic result of Lehmer \cite{Le39} of Rademacher type, similar to the proof of the log-concavity
of $p(n)$ \cite{Ni78, DP15}. Shortly after the result was published,
Alanazi, Gangola III,
and Munagi \cite{AGM17} came up with a subtle combinatorial proof.
Chern, Fu, and Tang \cite{CFT18} generalized and proved the
Bessenrodt--Ono inequality to $k$-colored partitions.
In \cite{HNT20} this was extended to
$k$ real, again involving polynomials. The proof was given by induction and
involving derivatives. Further, the work of Bessenrodt and Ono triggered the results of
Beckwith and Bessenrodt \cite{BB16} on $k$-regular partitions,
Hou and Jagadeesan \cite{HJ18} on the numbers of partitions with ranks in a given residue class modulo $3$ and 
Males \cite{Ma20} for general $t$,
and Heim and Neuhauser \cite{HN19}, and
Dawsey and Masri \cite{DM19} for the Andrews {\it{spt}}-function.

We performed several numerical experiments and are convinced that some
of the recorded results can be transferred to plane partitions $\func{pp}(n)$ and its generalization.
MacMahon \cite{Ma97,Ma99,Ma60} proved the following non-trivial result, which took him several years.
The generating function of the plane partition is given by
\[
\prod_{n=1}^{\infty} \left( 1- q^n \right)^{-n} = \sum_{n=0}^{\infty} \func{pp}(n) \, q^n.
\]
Since this generating function is not related to a weakly modular form, in contrast to the
partition numbers, we have only an asymptotic formula provided by Wright \cite{Wr31}
based on the circle and saddle point method 
for plane partition numbers. 
Wright proved the following asymptotic behavior as
$n$ goes to infinity:
\begin{equation}\label{Wright}
\func{pp}(n) \sim   \frac{\zeta(3)^{\frac{7}{36}}}{\sqrt{12 \pi}} 
\left( \frac{2}{n} \right)^{\frac{25}{36}} \exp
\left(  3 \, \zeta(3)^{\frac{1}{3}} \left( \frac{n}{2}\right)^{\frac{2}{3}}+ \zeta'(-1)\right).
\end{equation}
Here $\exp \left( z\right) = \mathrm{e}^z$
and $\zeta(s)$ denotes the Riemann zeta function.
Thus, we are in a similar situtation as described before by
DeSalvo and Pak \cite{DP15} for partition numbers.

Recently, we invented a new proof method \cite{HN21B} and reproved some known results related to the
Bessenrodt--Ono
inequality for the partition function, the $k$-colored
partitions and  extension to the
D'Arcais polynomials. 
\subsection{Main Results: Plane Partitions}
We first start with the Bessenrodt--Ono
inequality (\ref{bo}).
\begin{theorem}[Bessenrodt--Ono inequality] 
\label{BO: plane} Let $a$ and $b$ be positive integers. 
Let $a,b \geq 2$ and $a+b \geq 12$. Then
\begin{equation*}
\func{pp}(a) \, \func{pp}(b) > \func{pp}(a+b).
\end{equation*}
Equality is never satisfied.
\end{theorem}
Due to symmetry, let us assume that $2 \leq b \leq a$. Then $\func{pp}(a) \, \func{pp}(b) < \func{pp}(a+b)$
for
$(a,b) \in \left\{ \left( a,2\right) \, : \, 2 \leq a
\leq 9\right\} \, \cup \left\{ \left( a,3\right) \, : \, 2 \leq a
\leq 5\right\} $.
Note that $\func{pp}(n) < \func{pp}(n+1)$, similar to $p(n) < p(n+1)$. 
We have 
$$\left(\func{pp}(3)\right)^2 <  \func{pp}(3+3) 
\text{ and } 
\left(\func{pp}(4)\right)^2 > \func{pp}(4+4).$$
Based on our investigations we state the following
\begin{conjecture}\label{planeconjecture}
Let $n \geq 12$. Then the sequence $\{\func{pp}(n)\}_n$ of plane partitions is log-concave.
\begin{equation}
\func{pp}(n)^2 > \func{pp}(n-1) \, \func{pp}(n+1).
\label{eq:log-concave-plane}
\end{equation}
\end{conjecture}

We can show with Wright's formula 
(\ref{Wright}),
that
(\ref{eq:log-concave-plane}) is true for large $n$, 
and it seems that this is already true for all $n$ even and for all $n \geq 12$.
\begin{theorem}[Log-Concavity]
\label{asymptotic}Let $12 \leq n \leq 10^{5}$. Then Conjecture \ref{planeconjecture} is true.
It is further true for all $n$ even and false for all odd $n$ below $12$.
Furthermore there is an $N$ such that it is true for all $n>N$.
\end{theorem}
It would be very interesting to determine such a $N$ of reasonable size
and to finally prove the conjecture.
\subsection{Main Results: Polynomization}
It is possible to consider $\{\func{pp}(n)\}_n$ as
special values of a family of polynomials $\{P_n(x)\}_n$.
We will have $\func{pp}(n) = P_n(1)$.
This makes it possible to generalize the Bessenrodt--Ono
inequality and the log-concavity.
We view the inequalities as a property of the largest positive real zeros of new polynomials 
associated with $\{P_n(x)\}_n$.
\begin{definition}
Let $\sigma_2(n):= \sum_{d \mid n} d^2$. Let $P_0(x):=1$ and
\begin{equation}
P_n(x):= \frac{x}{n} \sum_{k=1}^n \sigma_2(k) \, P_{n-k}(x).
\label{xeasy}
\end{equation}
\end{definition}
We have listed the first polynomials in Table \ref{pnx}.
\begin{table}[H]
\[
\begin{array}{rl}
\hline
n&P_{n}\left( x\right) \\ \hline \hline
1&x
\\
2&\frac{1}{2}\*x^2
 + \frac{5}{2}\*x
\\
3&\frac{1}{6}\*x^3
 + \frac{5}{2}\*x^2
 + \frac{10}{3}\*x
\\
4&\frac{1}{24}\*x^4
 + \frac{5}{4}\*x^3
 + \frac{155}{24}\*x^2
 + \frac{21}{4}\*x
\\
5&\frac{1}{120}\*x^5
 + \frac{5}{12}\*x^4
 + \frac{115}{24}\*x^3
 + \frac{163}{12}\*x^2
 + \frac{26}{5}\*x
\\ \hline
\end{array}
\]
\caption{\label{pnx}Polynomials $P_{n}\left( x\right) $ for $n\in \left\{ 1,2,3,4,5\right\} $.}
\end{table}
Let $q,z \in \mathbb{C}$ and $\vert q \vert <1$. 
It is a standard procedure to show that
\begin{equation*}
\sum_{n=0}^{\infty} P_n(z) \, q^n =
\prod_{n=1}^{\infty} \left( 1 - q^n \right)^{- n \,z} =  \func{exp}\left( z \, \sum_{n=1}^{\infty} \sigma_2(n) \, \frac{q^n}{n} \right).
\end{equation*}
Thus, $n \, \func{pp}(n) = \sum_{k=1}^n \sigma_2(k) \, \func{pp}(n-k)$, applying MacMahon's discovery.
Let $k$ be a positive integer.
We would like to
call $P_n(k)$ the $k$-\emph{colored\/} plane partitions
(compare \cite{BBPT19}), but
at the moment there is no combinatorial interpretation available, as in the case of
partitions \cite{HNT20, BKRT21}.
The topic is quite
complicated, since MacMahon's result is already non-trivial and $\func{pp}(n)$
can also be identified with the number of all partitions of $n$, where each part $n_j$ is allowed to have $n_j$ colors.
\subsubsection{Bessenrodt--Ono Inequalities}
\begin{theorem}
\label{th:x}
Let $x \in \mathbb{R}$ and $x>5$. Then
\[
P_a(x) \, P_b(x) > P_{a+b}(x)
\]
for all positive integers $a$ and $b$.
\end{theorem}
It is also possible to get results for $x=1,2,3,4,5$.
This leads to restrictions on $a$ and $b$,
reflected in Table \ref{NS},
where we have recorded the largest real zero of
\begin{equation*}\label{pab}
P_{a,b}(x):= P_a(x) \, P_b(x) - P_{a+b}(x).
\end{equation*}
Thus, studying the polynomials $P_n(x)$ and $P_{a,b}(x)$ and
their leading coefficients and
zeros, provides the big picture and reveals information on the original task, studying properties
of plane partitions $\func{pp}(n)= P_n(1)$. Note
that $P_{a,b}(x)$ goes to infinity
for $a,b\geq 1$, as $x$ goes to infinity.
\begin{theorem} \label{th: bo2}
Let $x \in \mathbb{R}$ and $x\geq 2$. Then
\[
P_a(x) \, P_b(x) > P_{a+b}(x)
\]
for all positive integers $a$ and $b$ satisfying $a + b \geq 12$.
\end{theorem}

It would be interesting to search for a combinatorial proof
for the plane partition numbers (Theorem~\ref{BO: plane}) and their
generalization to $k$-colored plane partitions
(Theorems~\ref{th:x} and~\ref{th: bo2}).
\begin{table}[H]
\[
\begin{array}{rcccccccccccc}
\hline
a\backslash b&1&2&3&4&5&6&7&8&9&10&11&12\\ \hline \hline
1&5.0&3.2&3.0&2.6&2.5&2.3&2.2&2.1&2.1&2.0&2.0&1.9\\
2&3.2&1.9&1.6&1.4&1.3&1.2&1.1&1.1&1.0&1.0&1.0&0.9\\
3&3.0&1.6&1.5&1.2&1.1&1.0&1.0&0.9&0.9&0.8&0.8&0.8\\
4&2.6&1.4&1.2&0.9&0.9&0.8&0.7&0.7&0.6&0.6&0.6&0.6\\
5&2.5&1.3&1.1&0.9&0.9&0.7&0.7&0.6&0.6&0.6&0.6&0.5\\
6&2.3&1.2&1.0&0.8&0.7&0.6&0.6&0.5&0.5&0.5&0.5&0.4\\
7&2.2&1.1&1.0&0.7&0.7&0.6&0.6&0.5&0.5&0.5&0.5&0.4\\
8&2.1&1.1&0.9&0.7&0.6&0.5&0.5&0.4&0.4&0.4&0.4&0.4\\
9&2.1&1.0&0.9&0.6&0.6&0.5&0.5&0.4&0.4&0.4&0.4&0.3\\
10&2.0&1.0&0.8&0.6&0.6&0.5&0.5&0.4&0.4&0.4&0.3&0.3\\
11&2.0&1.0&0.8&0.6&0.6&0.5&0.5&0.4&0.4&0.3&0.3&0.3\\
12&1.9&0.9&0.8&0.6&0.5&0.4&0.4&0.4&0.3&0.3&0.3&0.3\\ \hline
\end{array}
\]
\caption{\label{NS}
Approximative largest real zeros of $
P_{a,b}\left( x\right) $ for $1\leq a,b\leq 12
$.}
\end{table}
\subsubsection{Log-Concavity}
In the spirit of the Chern--Fu--Tang
Conjecture \cite{CFT18} on $k$-colored partitions 
and their
polynomization \cite{HN21A} (see also \cite{BKRT21}), we consider 
the polynomials
\[
\Delta_{a,b}(x):= P_{a-1}(x)P_{b+1} (x) - P_{a}(x) \, P_b(x).
\]
Note that $\Delta_{a+1,a-1}(x)= P_a(x)^2 - P_{a-1}(x) \, P_{a+1}(x)$.
Then $\Delta_{a+1, a-1}(1)>0$
is the log-concavity condition for $\func{pp}(a)$.

\begin{minipage}{0.45\textwidth}
\includegraphics[width=\textwidth]{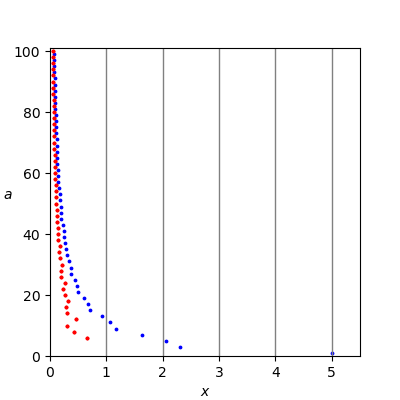}
\end{minipage}
\hfill
\begin{minipage}{0.52\textwidth}
\vspace{1.5cm}
\captionsetup{width=\linewidth}
\captionof{figure}{\label{ThePlot}
Zeros of $\Delta_{a+1, a-1}(x)$ with positive real part
for $ 1 \leq a \leq 100$.
Blue labels the real zeros and red the imaginary zeros.}
\end{minipage}
This leads to the following conjecture on the log-concavity of the polynomials $P_n(x)$.
For orthogonal polynomials, such kinds of inequalities
are called Tur\'{a}n inequality.
\begin{conjecture}[Tur\'{a}n inequality]\label{C1}
Let $a$ be an integer with $a \geq 12$ and let $x$ be a real number with $x\geq 1$.
Then 
\[
P_a(x)^2 > P_{a-1}(x) \, P_{a+1}(x).
\]
\end{conjecture}
This is a natural extension of Conjecture \ref{planeconjecture} on the log-concavity of
plane partitions.

It would also be interesting to study the hyperbolicity of the associated Jensen polynomials \cite{GORZ19}
and higher Tur\'{a}n inequalities \cite{CJW19}.
Let $\alpha(m):=\func{pp}(m)$ or more general, $\alpha(m):=P_m(x)$, $x \in \mathbb{R}_{>0}$.
The Jensen polynomial of degree $d$ and shift $n$ attached to the sequence $\{\alpha(0), \alpha(1), \alpha(2), \ldots\}$ of non-negative real numbers is the polynomial
\begin{equation*}
J_{\alpha}^{d,n}(X) := \sum_{k=0}^{d} \binom{d}{k} \, \alpha(n+k) \, X^k.
\end{equation*}
It follows from the zero distribution of the polynomial $\Delta_{a+1,a-1}(x):= P_a(x)^2 - P_{a-1}(x) \, P_{a+1}(x)$
(see Figure \ref{ThePlot}) that Conjecture \ref{C1} is true for $12 \leq a \leq 100$, since the polynomial goes to infinity,
as $x$ goes to infinity.
Let $a$ be even and $2 \leq a \leq
1000$, then the Tur\'{a}n inequality in Conjecture \ref{C1} is already
valid for $x >0$. Actually, in this case all coefficients of $\Delta_{a+1,a-1}(x)$ are non-negative.

The generalization of the former Chern--Fu--Tang conjecture
of $k$-colored partitions
and its polynomization \cite{HN21A} leads to:
\begin{conjecture}\label{C2}
Let $a$ and $b$ be integers. Suppose $ a-1 > b \geq 0$ and $(a,b) \not \in \{(4,2), (6,4)\}$. 
Then for all real numbers $x \geq 2$:
\[
\Delta_{a,b}(x):= P_{a-1}(x) \, P_{b+1}(x) - P_a(x) \, P_b(x) > 0.
\]
\end{conjecture}
We refer to Table \ref{REALNS}, where we recorded the largest real zero for $\Delta_{a,b}$ for
all pairs $(a,b)$ with $a-1 > b \geq 0$ and $2 \leq a \leq 19$,
which shows that the
Conjecture \ref{C2} is valid
for 
all admissible pairs $(a,b)$ in this range.
\section{Basic Properties of $\sigma_2(n)$ and $P_n(x)$}

For
proof of Theorem \ref{BO: plane} and Theorem \ref{th:x} we
need some
elementary properties of $\sigma_2(n)$ and $P_n(x)$.
As a further evidence that Conjecture \ref{C1} most likely
to be true for even arguments $n$, we
prove that $\sigma_2(n)/n$ is log-concave for $n$ even.
Note that this is false for $\frac{1}{n}\sum_{d \vert n} d$, where $n$ is even (which is related to
the Chern--Fu--Tang conjecture for $k$-colored partitions).
\begin{proposition}\label{lc:quotient}
Let $n$ be an even, positive integer. Then
\[
\left( \frac{\sigma_2(n)}{n} \right) ^{2}>  \frac{\sigma_2(n-1)}{n-1}  \,\, \frac{\sigma_2(n+1)}{n+1}.
\]
\end{proposition}
\begin{proof}
For $n$ odd we estimate
$$\sigma _{2}\left( n\right) =\sum _{t\mid n}t^{2}=
\sum _{t\mid n}\left( \frac{n
}{t
}\right) ^{2}=n^{2}\sum _{t\mid n}t^{-2}\leq n^{2}\left( 1-\frac{1}{4}\right) \sum _{t}t^{-2}\leq
\frac{\pi ^{2}}{8}n^{2}<\frac{5}{4}n^{2}.$$
Now, let $n\geq 6$ be even. Then
$\sigma _{2}\left( n\right) \geq n^{2}+\left( \frac{n}{2}\right) ^{2}+4
+1=\frac{5}{4}\left(
n^{2}
+4
\right) $.
Therefore,
$$\frac{\sigma _{2}\left( n-1\right) \sigma _{2}\left( n+1\right) }{n^{2}-1}<\frac{25}{16}
\left( n^{2}-1\right) <\left( \frac{5
}{4
}\frac{n^{2}+4}{n}\right) ^{2}\leq \left( \frac{\sigma _{2}\left( n\right) }{n}\right) ^{2}.$$
The inequality holds also for even $n\leq 4$
(Table~\ref{sigma2}).
\end{proof}
\begin{table}[H]
\[
\begin{array}{cccccc}
\hline
n & 1 & 2 & 3 & 4 & 5 \\ \hline \hline
\sigma _{2}\left( n\right) & 1 & 5 & 10 & 21 & 26 \\ \hline
\end{array}
\]
\caption{\label{sigma2}Values of $\sigma _{2}\left( n\right) $ for $n\in \left\{ 1,2,3,4,5\right\} $.}
\end{table}

Note, that
$\sigma_2(n) < \tilde{\sigma}_2(n)
:= 2 \, n^{2} 
$ for all $n \in \mathbb{N}$.
This follows from
$$\sum _{t\mid n}t^{2}=n^{2}\sum _{t\mid n}\frac{1}{t^{2}}\leq n^{2}
\left( 1+\int _{1}^{n}t^{-2}\,\mathrm{d}t\right) <2n^{2}.$$

The functions $P_n(x)$ are polynomials of degree $n$.
This
can be deduced directly from the recurrence formula (\ref{xeasy}).
We have
$P_n(x) = \frac{1}{n!} \, \sum_{m=1}^n A_{n,m} \, x^{m}
$, where $A_{n,m} \in \mathbb{N}$
for $n \geq 1$. We later use the fact that $A_{n,
1}= n! \, \frac{\sigma_2(n)}{n}$.
Further, we have the following properties.
\begin{proposition} Let $n,m$ be natural numbers and $x\geq 1$ real. Then
\begin{eqnarray}
P_n'(x) & = & \sum_{k=1}^n \frac{\sigma_{2}\left( k\right) }{k} \, P_{n-k}\left( x\right) ,
\label{derivative} \\
P_{n+1}(x) &> & P_{n}\left( x\right) ,\label{HNT Theorem}\\
P_n(x) &>& 
\sum_{\ell=1}^{m} \frac{1}{\ell !} \binom{n+ \ell -1}{2 \ell -1} \, x^{\ell} \text{ for } n>1.
\label{estimation}
\end{eqnarray}
\end{proposition}
\begin{proof}
The formula for the derivative $P_n'(x)$ is given similarly as that obtained for
the polynomials attached to $\sigma(n)= \sum_{d \mid n} d$ \cite{HN18}.
We prove (\ref{HNT Theorem}) by induction. Let
$\Delta _{n}\left( x\right) :=P_{n}\left( x\right) -P_{n-1}\left( x\right) $.

Let $n=1$. Then we have $\Delta_1(x)> 0$ for all $x>
1$. Suppose $\Delta_m(x)>0$ for all $1 \leq m \leq n-1$.
We obtain for $P_{n+1}'(x)$ the strict lower bound
\[
\sum_{k=1}^n \frac{\sigma_{2}\left( k\right) }{k} \, 
P_{n+1-k}(x) \geq \sum_{k=1}^n \frac{\sigma _{2}\left( k\right) }{k} \, P_{n-k}(x). 
\]
Thus, $P_{n+1}'(x) > P_{n}'(x)$. The plane partition function is strictly increasing:
\[
\func{pp}\left( n\right) < \func{pp}\left( n+1\right)
\]
for all $n \in \mathbb{N}$, since every plane
partition of $n$ can be lifted to a plane partition of $n+1$.
This provides
\[
P_{n+1}(1) = \func{pp}\left( n+1\right) > \func{pp}\left( n\right) = P_n (1).
\]
Thus, the claim is proven. 
The lower bound given in (\ref{estimation}) is obtained in the following way.
We have $n^2 < \sigma_2(n)$ for $n >1$ and $\sigma_2(1)=1$. Thus, the coefficients
of the polynomial $S_n(x)$,  defined by $S_0(x):=1$ and $S_n(x)= \frac{x}{n} \sum_{k=1}^n k^2 \, S_{n-k}(x)$,
are smaller than the coefficients of $P_n(x)$ for $n >1$. 
The $m$th coefficient is given by $
\frac{1}{
m !} \binom{n+
m -1}{2
m -1}
$.
This can be deduced from \cite{HLN19}, Section $4$.
\end{proof}
\begin{corollary} Let $n \in \mathbb{N}$ and $x \geq 1$. Then
$\Delta_n'(x) >0$.
\end{corollary}
\section{\label{Strategy}Proof Strategy for the Bessenrodt--Ono Inequality for Plane Partitions}

In this section we lay out a general strategy for proving Bessenrodt--Ono type inequalities.
This also makes the appearance of exceptions transparent.
\subsection{Input Data and Proof by Induction}
\label{AB}
Let $A,B,N_0 \in \mathbb{N}$ be given with $2 \leq B <N_0$ and $x_0 \in \mathbb{R}_{\geq 1}$.
Let $n \geq B$. Let $\mathrm{S}(n)$ be the following mathematical statement for one of the three cases:
$x \in \mathbb{R}$ with $x=x_0$, \, $x > x_0$, or $x \geq x_0$.
For all $A \leq b \leq a$ with $a+b =n$ and all $x$ with fixed case, we have
\begin{equation*}\label{statement}
P_{a,b}(x):= P_a(x) \, P_b(x) - P_{a+b}(x) > 0.
\end{equation*}
Note that $P_{a,b}(x_0)= P_{b,a}(x_0)$ and that $\lim_{x
\rightarrow \infty} P_{a,b}(x) = + \infty$.

Given input data $A,B,x_0$ by induction we prove $\mathrm{S}(n)$ for one of the given cases.
We choose $N_0$ and prove first manually
or by numerical calculation (utilizing PARI/GP) that
$\mathrm{S}\left( \ell \right) $ is true for all
$B \leq \ell \leq N_0$. In the case of $x \geq x_0$ or
$x >x_{0}$ we study the real 
zeros of the polynomial $P_{a,b}(x)$
for all $A \leq b \leq a$ with $a+b=\ell$.
Let $n \geq N_0$. Then we prove $\mathrm{S}(n)$ by assuming that $\mathrm{S}(m)$ is true for all $B \leq m \leq n-1$, the induction hypothesis.

\subsection{Basic Decomposition}
Let $A,B \in \mathbb{N}$ be given with $B \geq 2$. Further, let $a,b \in \mathbb{N}$ satisfy $A \leq b \leq a$ and $a+b \geq B$. We define
\begin{eqnarray*}
k_0 & := & a - \func{max}\{B-b, A \} +1,\\
f_k(a,b,x) & := & \sigma_2(k)
\left(
\frac{P_{a-k}(x) \, 
P_{b}(x)}{a} -
\frac{P_{a+b-k}(x)}{a+b}\right).
\end{eqnarray*}
We consider the decomposition $P_{a,b}(x) = L_{a,b}(x) + R_{a,b}(x)$, where
\begin{eqnarray*}\label{startl}
L=L_{a,b}(x)
&:=& -
\sum_{k=1}^{b}
\frac{\sigma_2 \left( k+a \right)}{a+b}
P_{b-k}\left( x\right) ,\\
R=R_{a,b}(x)
&:=& \sum _{k=1}^{a} f_{k}\left( a,b,x\right) .\nonumber
\end{eqnarray*}
Utilizing (\ref{HNT Theorem}) leads immediately to
\[
L > - 4 \, b \, a \, P_b(x).
\]
Further, let $R=R_1+ R_2 + R_3$, where
\begin{equation*}\label{decomposition R}
R_{1}:=
f_{1}\left( a,b,x\right)       , \,\, 
R_{2}:=\sum_{k=2}^{k_0-1} \,f_k(a,b,x) , \,\, 
R_{3}:=\sum_{k=k_0}^{a} \,f_k(a,b,x) .
\end{equation*}
Suppose $P_{a+b-k}(x) < P_{a-k} (x) \, P_b(x)$ for
$1 \leq k \leq k_0-1$, then
\[
R_1 > \frac{b}{2 \, a^2} P_{a-1}(x) \, P_b(x)  \text{ and } R_2 > 0.
\]
Further, we put $R_3 = R_{31} + R_{32} + R_{33}$, where
\begin{equation*}\label{decomposition R3}
R_{31}:=\sum_{k=k_0}^{a-A}  \,f_k(a,b,x)       , \,\, 
R_{32}:=\sum_{k=a-A+1}^{a-1} \,f_k(a,b,x) , \,\, 
R_{33}:=
f_{a}\left( a,b,x\right) .
\end{equation*}
Then $R_{33} >0$ and $R_{32}=0$ if $A=1$. Moreover, let $B-b \leq A$, then $R_{31}=0$.
Thus, $R_3$ can only have a negative contribution to $R$ if $A \geq 2$ (by $R_{32}$) and if $B-b>A$ (by $R_{31}$).
Note that $n^{2}\leq \sigma_2(n) < \tilde{\sigma}_2(n) := 2 \, n^2$.
We obtain by  straight-forward estimations the following.

\begin{lemma} Let $B-b > A$. Then
\[
R_{31}
>
\frac{a-A-k_0 +1}{a}\, \big(
k_{0}^{2} \, P_A(x) \, P_b(x) - \tilde{\sigma}_2(a-A) P_{a-b-k_0}(x) \big).
\]
If we know that some
$P_{k}\left( x\right) P_{b}\left( x\right) -P_{b+k}\left( x\right) $
is negative
we can use the estimate
\begin{equation}
R_{31}>\frac{\tilde{\sigma }_{2}\left( a
\right) }{a}\sum _{k=A}^{B-b}\left( P_{k}\left( x\right) P_{b}\left( x\right) -P_{k+b}\left( x\right) \right) )_{-}
\label{eq:minimal}
\end{equation}
where $v_{-}=v$ if $v<0$, otherwise $v_{-}=0$.
\end{lemma}
\section{Proof of Theorem \ref{th:x} and Theorem \ref{th: bo2}}
\label{both}
\begin{proof}[Proof of Theorem \ref{th:x}]
Let $x >5$. We prove that $P_{a,b}(x)>0$ for all $a,b \geq 1$. Due to symmetry, we
can assume $b \leq a$.
We follow the strategy presented in
Section \ref{Strategy}.
Let $A=1$ and $B=2$. Let $x_0 = 5$ and $x > x_0$. Note that $x=x_0$ does not work, since $P_{1,1}(5)=0$.
Let $N_0 = 12
$. Then the mathematical statement $\mathrm{S}(n)$ is true for all $A \leq a,b \leq N_0$ (see Table \ref{NS}).
Let $n > N_0$, we assume that $\mathrm{S}(m)$ is true for all $ B \leq m \leq n-1$. 
Let $a+b =n$ with $A \leq b \leq a$. Then 
$P_{a,b}(x) > L + R_1$. Note that $R_2, R_3 \geq 0$. With
(\ref{HNT Theorem}) we have
\begin{eqnarray*}
L & > & -4 \, a \, b \, P_{b}\left( x\right) , \\
R & > & \frac{b}{2 \, a^2} P_{a-1}(x) \, P_b(x).
\end{eqnarray*}
{\it Final step.} Putting things together and estimating $P_{a-1}(x)$ from below by (\ref{estimation}),
we obtain 
\begin{equation} \label{Grad7}
P_{a,b}(x) > \frac{b \, P_b(x)}{2 \, a^2}   \left(   - 8
\, a^3 + \sum_{\ell =1}^{5
} \binom{a + \ell -2}{ 2 \, \ell -1} \, \frac{5^{\ell}}{\ell!} \right).
\end{equation}
The right hand
side of (\ref{Grad7}) is polynomial in $a$ of degree $9$ with a positive leading
coefficient. Calculating the largest real zero shows that
the right hand
side is positive for $a \geq 7$. 
\end{proof}
\begin{proof}[Proof of Theorem \ref{th: bo2}]
Let $x\geq 2$. We put $A=1$ and $B=12$. Let $N_{0}=58$. Then by  induction, as before, we obtain
\[
P_{a,b}(x) >  L + R_1 + R_{31},
\]
where for $b\geq 12$ we
can apply the induction hypothesis and obtain $R_{31}>0$.
In case $1\leq b\leq 11$ we find from Table~\ref{NS} that
$P_{a-k,b}\left( 2\right) >0$ for
$b=1$ and $a-k\notin \left\{ 1,2,\ldots ,11\right\} $ or
$k=a-1$ and $b\notin \left\{ 1,2,\ldots ,11\right\} $.
Therefore,
$$R_{31}>\sum _{k=a-11}^{a-1}\frac{\sigma \left( k\right) P_{a-k,1}\left( x\right)
}{a}.$$ It can be checked that the polynomials
$P_{k,1}\left( x\right) $, $1\leq k\leq 10$, are monotonically increasing for
$x\geq 2$. Thus,
$$ R > R_{3
1}>-641
\cdot 2a\geq -641
\, a\, b\, P_{b}\left( x\right) $$
from (\ref{eq:minimal}) and Table~\ref{minimal}.

{\it Final step:}
Putting everything together leads to
\begin{equation}
P_{a,b} \left(
x\right)
>
\frac{ b \, P_{b} \left(
x\right)}{ 2 \, a^2}\left(
-1290a^{3}+
\sum_{\ell=1}^{8
}
\binom{a+\ell -2}{2\ell -1}\frac{
2
^{\ell }}{\ell !}\right)
. \label{x1.8final}
\end{equation}
In the last step we used the property (\ref{xeasy}) and that $x
\geq 2
$.
We obtain that the expression (\ref{x1.8final}) is positive for all $ a \geq 27
$. Since the leading
coefficient 
of
$P_{a,b}\left( x\right) $ is positive we only have to determine the largest real
zero
of all remaining $P_{a,b}\left( x\right)$. We checked this
for $12\leq b+
a\leq 52
$ with
PARI/GP (compare Table \ref{NS}).
\end{proof}
\begin{table}[H]
\[
\begin{array}{crrrrrrrrrr}
\hline
b&1&2&3&4&5&6&7&8&9&10\\ \hline \hline
&-641&-4&-11&-16&-38&-52&-101&-126&-180&-110\\ \hline
\end{array}
\]
\caption{\label{minimal}Values of $\sum
_{
k=1}^{
11-b}\left( P_{k}\left( 2\right) P_{b}\left( 2\right) -P_{k+b}\left( 2\right) \right) _{-}$ for $b\leq 10
$.}
\end{table}
\section{ Proof of the Bessenrodt--Ono Inequality: Theorem \ref{BO: plane}}
We start with the following auxiliary result.
\begin{lemma} \label{Stufe} We have $\func{pp}(2) = 3 \, \func{pp}(1)$ and for $n \neq 1$:
\begin{equation}
\func{pp}\left( n+1\right) < 3\func{pp}\left( n\right).
\label{eq:pp3}
\end{equation}
\end{lemma}

\begin{proof}
The proof is by
mathematical induction. We checked with PARI/GP that
(\ref{eq:pp3}) holds for $n\leq N_{0}=24
$.

Now let $n>N_{0}$. Then
\begin{eqnarray}
&&3\func{pp}\left( n\right) -\func{pp}\left( n+1\right) \nonumber \\
&=&\frac{3}{n}\sum _{k=1}^{n}\sigma _{2}\left( k\right) \func{pp}\left( n-k\right) -\frac{1
}{n+1}\sum _{k=1}^{n+1}\sigma _{2}\left( k\right) \func{pp}\left( n+1-k\right) \nonumber \\
&=&-\frac{\sigma _{2}\left( n+1\right) }{n+1}+\sum _{k=1}^{n}\sigma _{2}\left( k\right) \left( \frac{3}{n}
\func{pp}\left( n-k\right) -\frac{1}{n+1}\func{pp}\left( n-k+1\right) \right) \label{eq:IH} \\
&>&-2\frac{\left( n+1\right) ^{2}}{n+1}+\left( \frac{3
}{n}-\frac{3
}{n+1}\right) \func{pp}\left( n-1\right) \nonumber \\
&\geq &-2\left( n+1\right)
+\frac{3
}{\left( n+1\right) n}\sum _{\ell =1}^{3}\binom{n+\ell -2
}{2\ell -1}\frac{1}{\ell !}>0.\nonumber
\end{eqnarray}
We simplified (\ref{eq:IH}) by the induction hypothesis.
\end{proof}

\begin{proof}[Proof of Theorem~\ref{BO: plane}]
The proof is again by induction. We follow the proof strategy stated in
Section~\ref{Strategy}.
See also Section \ref{both}.

Let $A=2$, $B=12$ and $x=1$. Then
$k_{0}=a-\max \left\{ B-b,A\right\} +1=a-\max \left\{ 11-b,1\right\} $.
Furthermore, $L>-4a
b\func{pp}\left( b\right) $,
$R_{1}>\frac{b}{2a^{2}}\func{pp}\left( a-1\right) \func{pp}\left( b\right) $, and
$R_{2}>0$. For $R_{33}$ we obtain the lower bound $0$ and for $R_{32}$ we obtain
\[
\sigma _{2}\left( a-1\right) \left( \frac{\func{pp}\left( b\right) }{a}-\frac{\func{pp}\left( b+1\right) }{a+b}\right) \geq -\sigma _{2}\left( a-1\right) \frac{2\func{pp}\left( b\right) }{a}>-4a\func{pp}\left( b\right)
.
\]
Finally,
\begin{eqnarray*}
R_{31}&>&\sum _{k=k_{0}}^{a-A}\frac{\sigma \left( k\right) }{a}\left(
\func{pp}\left( a-k\right) \func{pp}\left( b\right) -\func{pp}\left( a+b-k\right) \right) \\
&\geq &-106
\cdot 2a
=-212
a
\end{eqnarray*}
(Table~\ref{pp}).
Therefore, note $R_{3}>-4a\func{pp}\left( b\right) -212
a>
-
38ab\func{pp}\left( b\right) $.
Putting everything together leads to
\begin{eqnarray*}
\func{pp}\left( a\right) \func{pp}\left( b\right) -\func{pp}\left( a+b\right) &>&\frac{b\func{pp}\left( b\right) }{2a^{2}}\left( -76
a^{3}+\func{pp}\left( a-1\right)
\right) \\
&>&\frac{b\func{pp}\left( b\right) }{2a^{2}}\left( -76
a^{3}+\sum _{\ell =1}^{3}
\frac{1}{\ell !}\binom{a-\ell -2}{2\ell -1}\right)
\end{eqnarray*}
using (\ref{xeasy}).
This is positive for $a\geq 237
$. We have checked with PARI/GP that
$\func{pp}\left( a\right) \func{pp}\left( b\right) -\func{pp}\left( a+b\right) >0$ for
$B\leq a+b\leq N_{0}=472
$.
\end{proof}

\begin{table}[H]
\[
\begin{array}{crrrrrrrr}
\hline
b&2&3&4&5&6&7&8&9\\ \hline \hline
&-106&-42&-17&-30&-16&-24&-20&-13\\ \hline
\end{array}
\]
\caption{\label{pp}Values of $\sum
_{
k=2}^{
11-b}\left( \func{pp}\left( k\right) \func{pp}\left( b\right) -\func{pp}\left( k+b\right) \right) $ for $2\leq b\leq 9$.}
\end{table}

\section{Proof of Theorem~\ref{asymptotic}}

\begin{lemma}
Let $s\in \mathbb{R}$. There are $C_{0},N>0$ such that
\begin{equation}
2n^{s}-\left( n+1\right) ^{s}-\left( n-1\right) ^{
s}=\left( 1-s\right) sn^{s-2}+E_{n}n^{s-3}
\label{eq:konkav}
\end{equation}
for $n>N$ with $\left| E_{n}\right| <C_{0}$.
\end{lemma}

\begin{proof}
We have
\begin{eqnarray*}
&&2n^{s}-\left( n-1\right) ^{s}-\left( n+1\right) ^{s}\\
&=&\left( 2-\left( 1-\frac{1}{n}\right) ^{s}-\left( 1+\frac{1}{n}\right) ^{s}\right) n^{s}\\
&=
&\left( 2-\left( 1-\frac{s}{n}-\frac{\left( 1-s\right) s}{2n^{2}}+\frac{
D
_{1,n}}{n^{3}}\right)
-\left( 1+
\frac{s}{n}-\frac{\left( 1-s\right) s}{2n^{2}}+\frac{
D_{2,n}}{n^{3}}\right) \right) n^{s}\\
&=&\left( 1-s\right) sn^{s-2}-
E_{n}
n^{s-3}
\end{eqnarray*}
with $\left| D_{1,n}\right| ,
\left| D_{2,n}\right| ,
\left| E_{n}\right| <C_{0}$
for some $C_{0},N
>0$ and all $n>N
$.
\end{proof}

\begin{corollary}
Let $C_{1
}>0$. There is $N>0$ such that
\begin{equation}
1+
\frac{C_{1
}
}{9}n^{-4/3}<\exp \left(
2C_{1
}n^{2/3}-C_{1
}\left( n+1\right) ^{2/3}-C_{1
}\left( n-1\right) ^{2/3}\right) <1+\frac{4C_{1
}}{9}n^{-4/3}
\label{eq:ungleichung}
\end{equation}
for all $n>N$.
\end{corollary}

\begin{proof}
For example for the lower bound we obtain from (\ref{eq:konkav})
\[
2n^{2/3}-\left( n-1\right) ^{2/3}-\left( n+1\right) ^{2/3}
\geq
\frac{2}{9}n^{-4/3}-2
C_{0}
n^{-7/3}>\frac{1}{9}
n^{-4/3}
\]
for some $C_{0}
>0$ and all $n>N$ for some $N
$. Therefore
\begin{eqnarray*}
\exp \left( 2C_{1
}n^{2/3}-C_{1
}\left( n-1\right) ^{2/3}-C_{1
}\left( n+1\right) ^{2/3}\right) &>&\exp \left( \frac{
C_{1
}}{9}n^{-4/3}\right) \\
&>
&1+\frac{
C_{1
}}{9}n^{-4/3}.
\end{eqnarray*}
\end{proof}

\begin{theorem}
Let $C_{1},C_{2},C_{3}>0$,
$
r,
\gamma _{1},
\gamma _{2}\in \mathbb{R}
$,
$\beta \left( n\right) =C_{2
}n^{r}\mathrm{e}^{C_{1
}n^{2/3}}$,
and
\[
\left| \frac{\alpha \left( n\right) }{
\beta \left( n\right) }-1-
\gamma _{1}n^{-2/3}-
\gamma _{2}n^{-4/3}
\right| \leq C_{3}n^{-2}
\]
for $n>N_{0}$ for some $N_{0}$.
There is an $N\geq N_{0}$ such that
$\left( \alpha \left( n\right) \right) _{n\geq N_{0}
}$ is log-concave for
$n>N
$.
\end{theorem}

\begin{proof}
Let $f_{\pm }\left( n\right) =
\gamma_{1}n^{-2/3}+
\gamma _{2}n^{-4/3}\pm C_{3}n^{-2}$. Then
there is an $N_{1}\geq N_{0}$ such that
\[
\left(
1+f_{-}\left(
n\right) 
\right) \beta \left( n\right) <\alpha \left( n\right) <\left(
1+f_{+}\left(
n\right)
\right) \beta \left( n\right)
\]
for all $n>N_{1}$. Therefore
\begin{eqnarray*}
&&\left( \alpha \left( n\right) \right) ^{2}-\alpha \left( n-1\right) \alpha \left( n+1\right) \\
&>&\left( \left(
1+f_{-}\left(
n\right) \right)
\beta \left( n\right) \right) ^{2}-\left(
1+f_{+}
\left( n-1\right)
\right)
\left(
1+f_{+}
\left( n+1\right)
\right)
\beta \left( n-1\right) \beta \left( n+1\right) \\
&=&\left(
1+f_{-}\left(
n
\right)
\right) ^{2}
\beta \left( n+1\right)
\beta \left( n-1\right)
\cdot {}\\
&&{}\cdot
( \exp \left( 2Bn^{2/3}-B\left( n-1\right) ^{2/3
}-B\left( n+1\right) ^{2/3
}\right) -{}\\
&&{}-
\frac{\left(
1+f_{+}\left( n-1
\right)
\right) \left(
1+f_{+}\left( n+1\right)
\right)
}{
\left(
1+f_{-}\left(
n\right)
\right) ^{2}
}\left( \frac{n^{2}-1}{n^{2}}\right) ^{
r}
)
.
\end{eqnarray*}
Now
\begin{eqnarray*}
&&\frac{\left( 1+f_{+}\left( n-1\right) \right) \left( 1+f_{+}\left( n+1\right) \right) }{\left( 1+f_{-}\left( n\right) \right) ^{2}}\\
&=&1+\frac{f_{+}\left( n-1\right) +f_{+}\left( n+1\right) -2f_{-}\left( n\right) }{1+f_{-}\left( n\right) }+{}\\
&&{}+\frac{\left( f_{+}\left( n-1\right) -f_{-}\left( n\right) \right) \left( f_{+}\left( n+1\right) -f_{-}\left( n\right) \right) }{\left( 1+f_{-}\left( n\right) \right) ^{2}}.
\end{eqnarray*}
With (\ref{eq:konkav}) we obtain
\[
\left| f_{+}\left( n-1\right) +f_{+}\left( n+1\right) -2f_{-}\left( n\right) \right| =\left| -\frac{10}{9}
\gamma _{1}n^{-8/3}
+E_{n}n^{-10/3}\right| <C
_{5
}n^{-8/3}
\]
with $\left| E_{n}\right| <C_{4}$, $C_{4}
,C_{5}>0$, and
\[
\left| \left( n+v\right) ^{u}-n^{u}\right| =\left|
1+\frac{uv}{n}
+\frac{
U_{n}}{n^{2}}-1\right| n^{u}=\left| uvn^{u-1}+
U_{n}n^{u-2}\right| <C_{6}n^{-5/3
}
\]
with $\left| U_{n}\right| <C_{
7
}$, $C_{
6},C_{7
}>0$,
and $u\leq -2/3
$.
Therefore
\[
\left| \frac{\left( 1+f_{+}\left( n-1\right) \right) \left( 1+f_{+}\left( n+1\right) \right) }{\left( 1+f_{-}\left( n\right) \right) ^{2}}\right| <1+C_{8}n^{-8/3
}
\]
for some $C_{8}>
C_{5}$
which implies with (\ref{eq:ungleichung})
\begin{eqnarray*}
&&\exp \left( 2C_{1
}n^{2/3}-C_{1
}\left( n-1\right) ^{2/3}-C_{1
}\left( n+1\right) ^{2/3}\right) -{}\\
&&{}
-
\frac{\left( 1+
f_{+}\left( n-1\right) \right) \left( 1+f_{+}\left( n+1\right) \right)
}{\left(
1+f_{-}\left(
n\right) \right) ^{2}}
\left(
1+\frac{1}{n^{2}-1}\right) ^{-r}\\
&>&1+
\frac{
C_{1
}}{9}n^{-4/3}
-
( 1+C_{8}n^{-8/3
}
) \left( 1+\frac{
C_{9}}{n^{2}
}\right) >0
\end{eqnarray*}
for some $C_{9}>0
$. This is positive
for $n>N$ and $N$ sufficiently large. 
\end{proof}

\begin{proof}[Proof of Theorem~\ref{asymptotic}]
Wright's formula (\cite{Wr31}, Formula~(2.21))
tells us that there are $C_{1},C_{2},C_{3}>0$,
$r=-25/36$, and
$\gamma _{1},
\gamma _{2}\in \mathbb{R}$
such
that
\[
\func{pp}\left( n\right) =
C_{2
}
n
^{-\frac{25}{36}} \mathrm{e}
^{C_{1
}
{n}
^{
{2}/{3}}}
\left( 1+\frac{\gamma _{1}}{n^{\frac{2}{3}}}+\frac{\gamma _{2}}{n^{\frac{4}{3}}}+\frac{G
_{n}
}{n^{
2}
}\right)
\]
with $\left| G
_{n}\right| <C_{3}$ for all $n>N_{0}$, $N_{0}$ sufficiently large.
\end{proof}
\section{Conjecture \ref{C1} and Conjecture \ref{C2}}
In this section we provide evidence for the Conjectures by 
information on the zeros of the underlying polynomials. Here we also use the crucial property that
the leading coefficient of these polynomials always has positive sign.
\subsection{Conjecture \ref{C1}}
Figure \ref{ThePlot} indicates that it is most likely that if one considers the sequence of real parts of
the zero of $\{\Delta_{a+1,a-1}(x)\}$ with the
largest real part, that these numbers tend in the limit to zero.
Another aspect is given by coefficients of these polynomials (Table \ref{examples}).
\begin{table}[H]
\[
\begin{array}{rl}
\hline
a&\Delta _{a+1,a-1}\left( x\right) \\ \hline \hline
2&\frac{1}{12}\*x^4
 + \frac{35}{12}\*x^2
\\
3&\frac{1}{144}\*x^6
 + \frac{5}{48}\*x^5
 + \frac{145}{144}\*x^4
 - \frac{101}{48}\*x^3
 - \frac{145}{72}\*x^2
\\
4&\frac{1}{2880}\*x^8
 + \frac{1}{72}\*x^7
 + \frac{67}{288}\*x^6
 + \frac{137}{144}\*x^5
 + \frac{11629}{2880}\*x^4
 + \frac{1373}{144}\*x^3
 + \frac{491}{48}\*x^2
\\
5&\frac{1}{86400}\*x^{10}
 + \frac{1}{1152}\*x^9
 + \frac{5}{192}\*x^8
 + \frac{941}{2880}\*x^7
 + \frac{63131}{28800}\*x^6
 + \frac{6449}{1152}\*x^5
 - \frac{1303}{864}\*x^4
 - \frac{54613}{1440}\*x^3
 - \frac{1671}{100}\*x^2
\\ \hline
\end{array}
\]
\caption{\label{examples}Polynomials $\Delta _{a+1,a-1}\left( x\right) $ for $a\in \left\{ 2,3,4,5\right\} $.}
\end{table}
We observe that for $a$ odd there are coefficients, which are negative.
In the case $a$ even, as already mentioned in the introduction, we have calculated
the polynomials for $2 \leq a \leq
1000$ and
observed that the coefficients are all non-negative. Let 
\[
\Delta_{a+1,a-1}\left( x\right) = \sum_{k=2}^{2a}
B_{2a,k}\, x^k,
\]
then we deduce from Proposition \ref{lc:quotient} that $B_{2a,2}>0$.
\newline
\newline
\newline
\
\subsection{Conjecture \ref{C2}}
Theorem \ref{th: bo2} implies that $\Delta_{a,0}(x)= P_{a-1}(x) \, P_1(x) - P_a(x) >0$ for
$a \geq 12$ and $x \geq 2$. Since $\Delta_{a,0}(x)> 0$ also for $3 \leq a \leq 11$ (Table \ref{REALNS}), we obtain:

\begin{corollary}
Let $x \geq 2$. Then Conjecture \ref{C2} is valid for all $a \geq 3$ and $b=0$. 
\end{corollary}
Next, we prove Conjecture \ref{C2} for $x \geq 1$ and all pairs $(a,1)$ with $a \geq 3$.
We refer also to Figure 
\ref{plot 1und2} for the pairs $(a,b)$ with $2\leq b \leq 4$. Note, for $(a,1)$ and $3 \leq a \leq 100$, there are no zeros
with a positive real part. 

\begin{minipage}{1.05\textwidth}
\begin{center}
\includegraphics[width=1.0\textwidth]{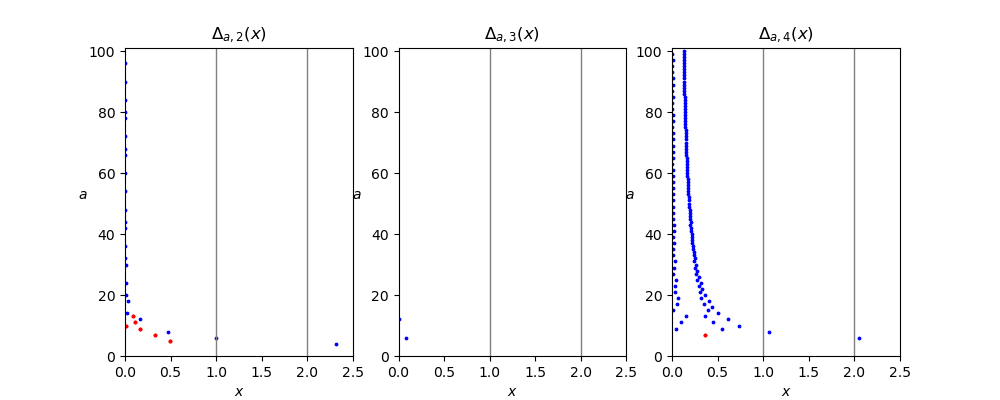}
\captionsetup{margin={0cm,0cm,0cm,0cm}}
\captionof{figure}{Zeros with the largest real part of $\Delta_{a,b}(x)$ for $ 2 \leq b \leq 4$,  \\
blue = real
zero, red = imaginary
zero.} \label{plot 1und2}
\end{center}
\end{minipage}
\begin{proposition}
Let $x \geq 1$ and $a \geq 3$. Then $\Delta_{a,1}(x)>0$.
\end{proposition}
\begin{proof}
We deduce from Table \ref{REALNS} that $\Delta_{a,1}(x)>0$ for 
$x >0$ and $a\leq 20$.
Moreover, $\Delta_{a,1}(1) = 3 \, \func{pp}(a-1)-\func{pp}(a) >0$ for $ a \geq 3$.
This follows from Lemma \ref{Stufe}.
Let $\Delta_{a,1}(x) = x \, F_a(x)$. We prove that $F_a(x)>0$ for $x \geq 1$ and $a\geq 3$ by induction on $a$.
Let $a \geq 6$ and $F_m(x) > 0$ for $3 \leq m < a$ and $x \geq 1$. We show that $F_a'(x)>0$, which completes the proof, 
since $F_a(1)> 1$, for $
a \geq 3
$. Recall Formula (\ref{derivative}). 
Then the derivative $F_a'(x)$ is equal to
\begin{equation*}
\frac{x+5}{2} \,
\sum_{k=1}^{a-1} \frac{\sigma_2(k)}{k} P_{a-1-k}(x) \,  - 
\sum_{k=1}^{a} \frac{\sigma_2(k)}{k} P_{a-k}(x)  + \frac{P_{a-1}(x)}{2}.
\end{equation*}
By the induction hypothesis we obtain
\begin{equation*} \frac{x+5}{2} 
\sum_{k=1}^{a-1} \frac{\sigma_2(k)}{k} P_{a-1-k}(x) \, > \, \frac{x+5}{2}
\frac{\sigma_2(a-1)}{a-1}  + 
\sum_{k=1}^{a-2} \frac{\sigma_2(k)}{k} P_{a-k}(x).
\end{equation*}
This leads to
\begin{eqnarray*}
F_{a}^{\prime }\left( x\right)  & > & \frac{x+5}{2}
\frac{\sigma_2(a-1)}{a-1}  - \sum_{k=a-1}^{a} \frac{\sigma_2(k)}{k} P_{a-k}(x)  + \frac{P_{a-1}(x)}{2}   \\
& = &
\frac{P_{a-1}(x)}{2} - \frac{x \, \sigma_2(a-1)}{2(a-1)}   +  \frac{5 \, \sigma_2(a-1)}{2\, (a-1)} - 
\frac{  \sigma_2(a)}{a}.
\end{eqnarray*}
Moreover, we have
\begin{equation*}
2 \, F_a'(x) > \left( P_{a-1}(x) - \frac{\sigma_2(a-1)}{a-1} \, x \right) + a-5,
\end{equation*}
since $a^2 < \sigma_2(a) < 2 \, a^2$. We observe that 
$P_{a}(x)- \frac{\sigma_2(a)}{a}\, x$ has
non-negative coefficients. Finally, this implies that $F_a'(x)>0$ for all $x>1$.
Thus, $F_a(x)>0$ since $F_a(1)>0$.
\end{proof}
\begin{corollary}
Let $x \geq 1$ and $a \geq 3$. Then $\Delta_{a,1}'(x)>0$.
\end{corollary}
\begin{remarks} \ \\
a) The real part of the zeros of $\Delta_{a,1}(x)/ \, x $ is negative for $3 \leq a \leq 100$.\\
b) The method of the proof is similar to the one outlined in \cite{HN21A}.
\end{remarks}
\begin{table}
\resizebox{\columnwidth}{!}{
\begin{tabular}{|r|ccccccccccccccccccc}
\cline{1 - 2}
   2 & \multicolumn{1}{c|}{\textbf{5.00}} & \hphantom{0.00}  & \hphantom{0.00}  & \hphantom{0.00}  & \hphantom{0.00}  & \hphantom{0.00}  & \hphantom{0.00}  & \hphantom{0.00}  & \hphantom{0.00}  & \hphantom{0.00}  & \hphantom{0.00}  & \hphantom{0.00}  & \hphantom{0.00}  & \hphantom{0.00}  & \hphantom{0.00}  & \hphantom{0.00}  & \hphantom{0.00}  & \hphantom{0.00}  & \hphantom{0.00}  \\
\cline{3 - 3}   3 & \textbf{3.16} & \multicolumn{1}{c|}{\hphantom{-}--\hphantom{-} } & \hphantom{0.00}  & \hphantom{0.00}  & \hphantom{0.00}  & \hphantom{0.00}  & \hphantom{0.00}  & \hphantom{0.00}  & \hphantom{0.00}  & \hphantom{0.00}  & \hphantom{0.00}  & \hphantom{0.00}  & \hphantom{0.00}  & \hphantom{0.00}  & \hphantom{0.00}  & \hphantom{0.00}  & \hphantom{0.00}  & \hphantom{0.00}  & \hphantom{0.00}  \\
\cline{4 - 4}   4 & \textbf{3.00} & \hphantom{-}--\hphantom{-}  & \multicolumn{1}{c|}{\textbf{2.31}} & \hphantom{0.00}  & \hphantom{0.00}  & \hphantom{0.00}  & \hphantom{0.00}  & \hphantom{0.00}  & \hphantom{0.00}  & \hphantom{0.00}  & \hphantom{0.00}  & \hphantom{0.00}  & \hphantom{0.00}  & \hphantom{0.00}  & \hphantom{0.00}  & \hphantom{0.00}  & \hphantom{0.00}  & \hphantom{0.00}  & \hphantom{0.00}  \\
\cline{5 - 5}   5 & \textbf{2.57} & \hphantom{-}--\hphantom{-}  & \hphantom{-}--\hphantom{-}  & \multicolumn{1}{c|}{\hphantom{-}--\hphantom{-} } & \hphantom{0.00}  & \hphantom{0.00}  & \hphantom{0.00}  & \hphantom{0.00}  & \hphantom{0.00}  & \hphantom{0.00}  & \hphantom{0.00}  & \hphantom{0.00}  & \hphantom{0.00}  & \hphantom{0.00}  & \hphantom{0.00}  & \hphantom{0.00}  & \hphantom{0.00}  & \hphantom{0.00}  & \hphantom{0.00}  \\
\cline{6 - 6}   6 & \textbf{2.50} & \hphantom{-}--\hphantom{-}  & 1.00 & 0.09 & \multicolumn{1}{c|}{\textbf{2.05}} & \hphantom{0.00}  & \hphantom{0.00}  & \hphantom{0.00}  & \hphantom{0.00}  & \hphantom{0.00}  & \hphantom{0.00}  & \hphantom{0.00}  & \hphantom{0.00}  & \hphantom{0.00}  & \hphantom{0.00}  & \hphantom{0.00}  & \hphantom{0.00}  & \hphantom{0.00}  & \hphantom{0.00}  \\
\cline{7 - 7}   7 & \textbf{2.30} & \hphantom{-}--\hphantom{-}  & \hphantom{-}--\hphantom{-}  & \hphantom{-}--\hphantom{-}  & \hphantom{-}--\hphantom{-}  & \multicolumn{1}{c|}{\hphantom{-}--\hphantom{-} } & \hphantom{0.00}  & \hphantom{0.00}  & \hphantom{0.00}  & \hphantom{0.00}  & \hphantom{0.00}  & \hphantom{0.00}  & \hphantom{0.00}  & \hphantom{0.00}  & \hphantom{0.00}  & \hphantom{0.00}  & \hphantom{0.00}  & \hphantom{0.00}  & \hphantom{0.00}  \\
\cline{8 - 8}   8 & \textbf{2.25} & \hphantom{-}--\hphantom{-}  & 0.47 & \hphantom{-}--\hphantom{-}  & 1.06 & \hphantom{-}--\hphantom{-}  & \multicolumn{1}{c|}{1.64} & \hphantom{0.00}  & \hphantom{0.00}  & \hphantom{0.00}  & \hphantom{0.00}  & \hphantom{0.00}  & \hphantom{0.00}  & \hphantom{0.00}  & \hphantom{0.00}  & \hphantom{0.00}  & \hphantom{0.00}  & \hphantom{0.00}  & \hphantom{0.00}  \\
\cline{9 - 9}   9 & \textbf{2.14} & \hphantom{-}--\hphantom{-}  & \hphantom{-}--\hphantom{-}  & \hphantom{-}--\hphantom{-}  & 0.55 & \hphantom{-}--\hphantom{-}  & 0.73 & \multicolumn{1}{c|}{\hphantom{-}--\hphantom{-} } & \hphantom{0.00}  & \hphantom{0.00}  & \hphantom{0.00}  & \hphantom{0.00}  & \hphantom{0.00}  & \hphantom{0.00}  & \hphantom{0.00}  & \hphantom{0.00}  & \hphantom{0.00}  & \hphantom{0.00}  & \hphantom{0.00}  \\
\cline{10 - 10}  10 & \textbf{2.09} & \hphantom{-}--\hphantom{-}  & \hphantom{-}--\hphantom{-}  & \hphantom{-}--\hphantom{-}  & 0.73 & \hphantom{-}--\hphantom{-}  & 0.89 & \hphantom{-}--\hphantom{-}  & \multicolumn{1}{c|}{1.17} & \hphantom{0.00}  & \hphantom{0.00}  & \hphantom{0.00}  & \hphantom{0.00}  & \hphantom{0.00}  & \hphantom{0.00}  & \hphantom{0.00}  & \hphantom{0.00}  & \hphantom{0.00}  & \hphantom{0.00}  \\
\cline{11 - 11}  11 & \textbf{2.03} & \hphantom{-}--\hphantom{-}  & \hphantom{-}--\hphantom{-}  & \hphantom{-}--\hphantom{-}  & 0.45 & \hphantom{-}--\hphantom{-}  & 0.55 & \hphantom{-}--\hphantom{-}  & \hphantom{-}--\hphantom{-}  & \multicolumn{1}{c|}{\hphantom{-}--\hphantom{-} } & \hphantom{0.00}  & \hphantom{0.00}  & \hphantom{0.00}  & \hphantom{0.00}  & \hphantom{0.00}  & \hphantom{0.00}  & \hphantom{0.00}  & \hphantom{0.00}  & \hphantom{0.00}  \\
\cline{12 - 12}  12 & 1.99 & \hphantom{-}--\hphantom{-}  & 0.16 & 0.00 & 0.62 & \hphantom{-}--\hphantom{-}  & 0.71 & 0.04 & 0.69 & 0.28 & \multicolumn{1}{c|}{1.07} & \hphantom{0.00}  & \hphantom{0.00}  & \hphantom{0.00}  & \hphantom{0.00}  & \hphantom{0.00}  & \hphantom{0.00}  & \hphantom{0.00}  & \hphantom{0.00}  \\
\cline{13 - 13}  13 & 1.94 & \hphantom{-}--\hphantom{-}  & \hphantom{-}--\hphantom{-}  & \hphantom{-}--\hphantom{-}  & 0.36 & \hphantom{-}--\hphantom{-}  & 0.46 & \hphantom{-}--\hphantom{-}  & \hphantom{-}--\hphantom{-}  & \hphantom{-}--\hphantom{-}  & \hphantom{-}--\hphantom{-}  & \multicolumn{1}{c|}{\hphantom{-}--\hphantom{-} } & \hphantom{0.00}  & \hphantom{0.00}  & \hphantom{0.00}  & \hphantom{0.00}  & \hphantom{0.00}  & \hphantom{0.00}  & \hphantom{0.00}  \\
\cline{14 - 14}  14 & 1.91 & \hphantom{-}--\hphantom{-}  & 0.02 & \hphantom{-}--\hphantom{-}  & 0.51 & \hphantom{-}--\hphantom{-}  & 0.57 & \hphantom{-}--\hphantom{-}  & 0.48 & 0.06 & 0.59 & \hphantom{-}--\hphantom{-}  & \multicolumn{1}{c|}{0.92} & \hphantom{0.00}  & \hphantom{0.00}  & \hphantom{0.00}  & \hphantom{0.00}  & \hphantom{0.00}  & \hphantom{0.00}  \\
\cline{15 - 15}  15 & 1.88 & \hphantom{-}--\hphantom{-}  & \hphantom{-}--\hphantom{-}  & \hphantom{-}--\hphantom{-}  & 0.39 & \hphantom{-}--\hphantom{-}  & 0.46 & \hphantom{-}--\hphantom{-}  & 0.27 & \hphantom{-}--\hphantom{-}  & 0.33 & \hphantom{-}--\hphantom{-}  & 0.46 & \multicolumn{1}{c|}{\hphantom{-}--\hphantom{-} } & \hphantom{0.00}  & \hphantom{0.00}  & \hphantom{0.00}  & \hphantom{0.00}  & \hphantom{0.00}  \\
\cline{16 - 16}  16 & 1.85 & \hphantom{-}--\hphantom{-}  & \hphantom{-}--\hphantom{-}  & \hphantom{-}--\hphantom{-}  & 0.44 & \hphantom{-}--\hphantom{-}  & 0.49 & \hphantom{-}--\hphantom{-}  & 0.37 & \hphantom{-}--\hphantom{-}  & 0.43 & \hphantom{-}--\hphantom{-}  & 0.55 & \hphantom{-}--\hphantom{-}  & \multicolumn{1}{c|}{0.70} & \hphantom{0.00}  & \hphantom{0.00}  & \hphantom{0.00}  & \hphantom{0.00}  \\
\cline{17 - 17}  17 & 1.82 & \hphantom{-}--\hphantom{-}  & \hphantom{-}--\hphantom{-}  & \hphantom{-}--\hphantom{-}  & 0.35 & \hphantom{-}--\hphantom{-}  & 0.41 & \hphantom{-}--\hphantom{-}  & 0.19 & \hphantom{-}--\hphantom{-}  & 0.25 & \hphantom{-}--\hphantom{-}  & 0.33 & \hphantom{-}--\hphantom{-}  & \hphantom{-}--\hphantom{-}  & \multicolumn{1}{c|}{\hphantom{-}--\hphantom{-} } & \hphantom{0.00}  & \hphantom{0.00}  & \hphantom{0.00}  \\
\cline{18 - 18}  18 & 1.80 & \hphantom{-}--\hphantom{-}  & 0.03 & \hphantom{-}--\hphantom{-}  & 0.40 & \hphantom{-}--\hphantom{-}  & 0.45 & \hphantom{-}--\hphantom{-}  & 0.33 & 0.05 & 0.37 & \hphantom{-}--\hphantom{-}  & 0.44 & 0.04 & 0.43 & 0.18 & \multicolumn{1}{c|}{0.68} & \hphantom{0.00}  & \hphantom{0.00}  \\
\cline{19 - 19}  19 & 1.78 & \hphantom{-}--\hphantom{-}  & \hphantom{-}--\hphantom{-}  & \hphantom{-}--\hphantom{-}  & 0.32 & \hphantom{-}--\hphantom{-}  & 0.38 & \hphantom{-}--\hphantom{-}  & 0.17 & \hphantom{-}--\hphantom{-}  & 0.22 & \hphantom{-}--\hphantom{-}  & 0.29 & \hphantom{-}--\hphantom{-}  & \hphantom{-}--\hphantom{-}  & \hphantom{-}--\hphantom{-}  & \hphantom{-}--\hphantom{-}  & \multicolumn{1}{c|}{\hphantom{-}--\hphantom{-} } & \hphantom{0.00}  \\
\cline{20 - 20}  20 & 1.76 & \hphantom{-}--\hphantom{-}  & 0.01 & \hphantom{-}--\hphantom{-}  & 0.36 & \hphantom{-}--\hphantom{-}  & 0.41 & \hphantom{-}--\hphantom{-}  & 0.29 & 0.03 & 0.32 & \hphantom{-}--\hphantom{-}  & 0.37 & 0.01 & 0.31 & 0.08 & 0.41 & \hphantom{-}--\hphantom{-}  & \multicolumn{1}{c|}{0.61} \\
\hline
 $x_{a,b}$  &    0 &    1 &    2 &    3 &    4 &    5 &    6 &    7 &    8 &    9 &   10 &   11 &   12 &   13 &   14 &   15 &   16 &   17 &   \multicolumn{1}{c|}{  18} \\
\hline
\end{tabular}
}
\caption{\label{REALNS}
Approximative
largest positive real zeros of $
\Delta_{a,b}\left( x\right) $ for $0 \leq b < a-1 \leq 19
$.}
\end{table}


\end{document}